\pdfoutput=1
\documentclass[11pt,reqno]{amsart}

\usepackage[a4paper,margin=1.15in]{geometry}
\usepackage[T1]{fontenc}
\usepackage{lmodern}
\usepackage{microtype}
\usepackage{mathtools,amssymb}
\usepackage[numbers,sort&compress]{natbib}
\usepackage{xcolor}
\usepackage[colorlinks=true,linkcolor=blue!50!black,citecolor=blue!50!black,urlcolor=blue!50!black]{hyperref}

\newcommand{\ECT}{\mathrm{ECT}}
\newcommand{\Prob}{\mathbb{P}}

\newcommand{\R}{\mathbb{R}}
\newcommand{\Sphere}{\mathbb{S}}

\newcommand{\norm}[1]{\left\lVert #1 \right\rVert}
\newcommand{\card}[1]{\left\lvert #1 \right\rvert}

\theoremstyle{plain}
\newtheorem{theorem}{Theorem}[section]

\newtheorem{proposition}[theorem]{Proposition}
\newtheorem{corollary}[theorem]{Corollary}
\theoremstyle{remark}
\newtheorem{remark}[theorem]{Remark}

\begin{document}

\title[Probabilistic ECT Stability Transfer]{Probabilistic Stability Transfer for\\One-Dimensional Euler Characteristic Transforms}

\author{Vanja Stojanovi\'c}
\address{Faculty of Mathematics and Physics, University of Ljubljana, Jadranska ulica 21, SI-1000 Ljubljana, Slovenia}
\email{vs66277@student.uni-lj.si}


\begin{abstract}
    The Euler characteristic transform (ECT) is a computable topological signature for shapes embedded in Euclidean space. Existing ECT stability results control changes in the transform through deterministic geometric perturbation bounds, but many data-generating mechanisms produce perturbations that are random and unbounded in the worst case. In such settings a uniform deterministic bound may be unavailable or too pessimistic. This paper formulates a probabilistic stability-transfer principle for one-dimensional ECT stability: a deterministic stability estimate is combined with a tail bound for the induced geometric error to obtain a high-probability ECT error bound. We instantiate this principle using the stability theorem of Marsh and Beers for embeddings of finite one-dimensional CW complexes. The resulting consequences cover ambient displacement fields conditionally on admissibility and supremum-tail estimates, and finite landmark perturbations with reconstruction. For landmark perturbations, bounded, sub-Gaussian, sub-exponential, and polynomial-tail noise assumptions give explicit high-probability ECT stability bounds. The framework isolates a geometric stochastic component that arises in shape analysis pipelines across biological morphology, medical imaging, molecular structure analysis, and remote sensing.
\end{abstract}
\maketitle

\section{Introduction}

The Euler characteristic transform (ECT) is a topological signature that summarizes a compact shape by recording Euler characteristics of directional sublevel sets. It is computationally attractive and is known to be injective on broad classes of embedded shapes; see \cite{Munch02012025} for an accessible introduction. The ECT is closely related to the persistent homology transform introduced by Turner, Mukherjee, and Boyer \cite{TurnerMukherjeeBoyer2014PHT}, and its injectivity can be understood through Euler calculus and Euler integral transforms \cite{GhristLevangerMai2018EulerIntegral,CurryMukherjeeTurner2022Directions}. Recent work has also established stability and statistical inference results for the ECT in the setting of compact one-dimensional shapes, with curvature playing a central role in the stability estimate \cite{marsh-lewis-beers-davis-1d-ect}.

The practical motivation is broader than any single imaging modality. ECT and closely related Euler characteristic summaries have been used for biological morphology, including barley seed shape analysis \cite{AmezquitaQuigleyOpheldersLandisKoenigMunchChitwood2022Barley}, medical imaging through the SECT analysis of glioblastoma tumors \cite{CrawfordMonodChenMukherjeeRabadan2020SECT}, and protein structure or molecular-dynamics analysis through Euler-characteristic-based topological summaries \cite{TangDaSilvaKirveslahtiSkeensFengSudijonoYangMukherjeeRubensteinCrawford2022Protein}. More generally, Euler characteristic profiles and related transforms have been studied as efficient topological descriptors for machine-learning tasks \cite{HacquardLebovici2024EulerTools}, and recent digital ECT tools emphasize shape comparison without relying on landmarks or diffeomorphism assumptions \cite{KirveslahtiWang2025DigitalECT}. These applications share a common issue: the object whose topology is analyzed is often produced from noisy measurements, segmentation, reconstruction, or registration.

The goal of this paper is to isolate a general probabilistic question that appears whenever ECTs are computed from noisy one-dimensional geometric data. Deterministic stability theorems typically control the ECT error by a deterministic geometric perturbation size. This is the correct language for worst-case analysis, but it is often too strong for data generated by random mechanisms. Gaussian, sub-exponential, and heavy-tailed perturbations are not uniformly bounded. A deterministic worst-case radius can therefore be infinite or practically useless, even when large perturbations have small probability.

The central point is simple: if a deterministic theorem bounds the ECT error by a function of a geometric error radius, then any probabilistic tail bound for that radius gives a high-probability ECT stability bound. The contribution of this paper is to formulate this transfer carefully, separate the deterministic and probabilistic assumptions, and instantiate the framework in a one-dimensional setting using the stability theorem of Marsh and Beers. The resulting argument is modular: the stochastic part changes with the perturbation model, while the deterministic ECT stability input can be imported as a theorem.

One long-term application direction is synthetic aperture radar (SAR). SAR imagery is commonly modeled as being corrupted by multiplicative speckle noise, and homomorphic despeckling methods use a logarithmic transform to convert multiplicative noise into additive noise \cite{Choi2019Speckle,Garg2023DESPECKLING,Singh2018A}. However, raw SAR images and range-Doppler maps raise additional modeling questions: the input may be a grayscale image, a cubical complex, an excursion-set filtration, or an extracted geometric object. This paper does not prove an ECT theorem for SAR images directly. Instead, it develops a geometric probabilistic stability component for one-dimensional extracted objects; for SAR signal-processing background see \cite{alma991000938109707381}.

\paragraph{Contributions.}
The paper is structured around the following contributions:
\begin{itemize}
    \item an abstract probabilistic transfer theorem converting deterministic ECT stability into high-probability ECT stability;
    \item a one-dimensional instantiation based on the Marsh--Beers deterministic stability theorem;
    \item ambient displacement and finite landmark perturbation models for one-dimensional embedded shapes;
    \item explicit landmark-noise corollaries for bounded, sub-Gaussian, sub-exponential, and polynomial-tail regimes;
    \item a precise boundary between this geometric theory and future SAR image-domain stability questions.
\end{itemize}

\section{Background}

\subsection{Euler characteristic transform}

Let \(K \subset \R^d\) be a compact tame set, and let \(v \in \Sphere^{d-1}\) be a direction. The directional height function is
\[
    h_v(x) = \langle x,v\rangle.
\]
The Euler characteristic transform of \(K\) is the function
\[
    \ECT_K(v,t)
    =
    \chi\bigl(K \cap h_v^{-1}((-\infty,t])\bigr),
    \qquad
    (v,t)\in \Sphere^{d-1}\times \R.
\]
Here \(\chi\) denotes Euler characteristic. Intuitively, \(\ECT_K(v,t)\) records how the topology of \(K\) appears as \(K\) is swept by half-spaces orthogonal to \(v\).

\subsection{Related topological transforms}

The ECT sits in a family of directional topological transforms. The persistent homology transform (PHT) records persistence diagrams of directional filtrations and was introduced as an injective statistic for modeling shapes and surfaces \cite{TurnerMukherjeeBoyer2014PHT}. The ECT replaces persistence diagrams by Euler characteristic curves, making it computationally simpler while retaining strong reconstruction properties on tame or constructible shape classes. Ghrist, Levanger, and Mai interpret this phenomenon through Euler integral transforms \cite{GhristLevangerMai2018EulerIntegral}, while Curry, Mukherjee, and Turner prove sufficiency results showing when finitely many directions determine a shape \cite{CurryMukherjeeTurner2022Directions}.

For statistical applications, the smooth Euler characteristic transform (SECT) integrates Euler characteristic curves to obtain functional summaries with better linear and functional-analytic structure. Crawford et al. used the SECT for clinical-outcome prediction from glioblastoma images \cite{CrawfordMonodChenMukherjeeRabadan2020SECT}. More recently, Meng et al. developed a probability-space framework for random shapes represented through the SECT \cite{MengWangCrawfordEloyan2024RandomSECT}. These statistical developments are adjacent to the present paper, but our focus is different: we seek direct high-probability stability statements for ECT errors induced by random geometric perturbations.

The emphasis on stability is consistent with a broader theme in topological data analysis. Classical persistence-diagram stability results show that small perturbations of suitable functions lead to controlled perturbations of persistence summaries \cite{CohenSteinerEdelsbrunnerHarer2005Stability}. Probabilistic work in TDA often studies distributions of topological summaries, statistical inference, or random-shape models. The present paper takes a narrower route: it treats a deterministic ECT stability theorem as an input and records how tail estimates for the induced geometric perturbation transfer into high-probability ECT bounds.

\section{Deterministic stability input}

The framework assumes a deterministic ECT stability statement of the following abstract form. Let \(\mathcal{C}\) be a class of compact tame sets, let \(d_{\mathrm{geom}}\) be a geometric distance on \(\mathcal{C}\), and let \(d_{\ECT}\) be a distance between ECTs. Suppose there exists a nondecreasing function \(\Phi:[0,\infty)\to[0,\infty)\) such that
\begin{equation}
    d_{\ECT}(K,L)
    \le
    \Phi\bigl(d_{\mathrm{geom}}(K,L)\bigr)
    \label{eq:deterministic-stability}
\end{equation}
for all \(K,L\in \mathcal{C}\). This gives a natural way to compare geometric difference between shapes with functional difference between their ECTs. The function \(\Phi\) is a stability modulus: it specifies how geometric error is converted into ECT error.

This paper does not assert that \eqref{eq:deterministic-stability} holds for every reasonable shape class or every geometric metric. The objects \(\mathcal{C}\), \(d_{\mathrm{geom}}\), \(d_{\ECT}\), and \(\Phi\) must come from a deterministic theorem or be proved separately. For the first concrete instantiation we use the one-dimensional embedded-shape stability theorem of Marsh and Beers \cite[Theorem~2]{marsh-lewis-beers-davis-1d-ect}. This choice is compatible with pipelines that extract contours, skeletons, graphs, or target boundary curves. It is not, by itself, a theorem for raw grayscale SAR images, range-Doppler maps, masks, or cubical complexes.

\subsection{The Marsh--Beers one-dimensional input}

Let \(Z\) be a finite one-dimensional CW complex with fixed CW structure
\[
    Z^*=(Z,Z_0,\{\Phi_\lambda\}_{\lambda\in\Lambda}).
\]
For \(r\geq 1\), let \(\mathcal G^r(Z^*,d)\) denote the class of subsets \(X\subseteq\R^d\) that arise as images of \(C^r\) homeomorphisms from \(Z\) into \(\R^d\), in the sense of Marsh--Beers. They define a metric \(d_{Z^*}\) on \(\mathcal G^r(Z^*,d)\) by controlling both the uniform distance between compatible parametrizations and the change in arc length on each one-cell. They also use the ECT norm
\[
    \norm{\ECT_X}
    =
    \sup_{v\in \Sphere^{d-1}}
    \int_{-a}^{a}\left|\ECT_X(v,t)\right|\,dt,
\]
where \(a\) is chosen large enough that the shapes under consideration lie in the Euclidean ball of radius \(a\).

\begin{theorem}[Marsh--Beers \citealp{marsh-lewis-beers-davis-1d-ect}, Theorem~2]
\label{thm:marsh-beers-input}
Let \(Z^*\) be a fixed finite one-dimensional CW structure and let \(r\geq 2\). The map
\[
    X\longmapsto \ECT_X
\]
is continuous on \(\mathcal G^r(Z^*,d)\) equipped with the metric \(d_{Z^*}\). More precisely, if \(X,Y\in\mathcal G^r(Z^*,d)\), if \(X\) has curvature bounded by \(M\), and if the image of the \(\lambda\)-th one-cell in \(X\) has arc length \(L_\lambda\), then whenever \(d_{Z^*}(X,Y)<\varepsilon\),
\[
    \norm{\ECT_X-\ECT_Y}
    \le
    \Phi_{MB}(\varepsilon)
    :=
    \card{Z_0}\varepsilon+\sum_{\lambda\in\Lambda}G_\lambda(\varepsilon),
\]
where
\[
G_\lambda(\varepsilon)
:=
\begin{cases}
8\sqrt{L_\lambda n_\lambda\varepsilon}+n_\lambda\varepsilon,
& L_\lambda/n_\lambda>2\varepsilon,\\
11n_\lambda\varepsilon,
& L_\lambda/n_\lambda\leq 2\varepsilon,
\end{cases}
\]
and
\[
    n_\lambda
    :=
    \max\left\{
        \left\lceil
        \left(\frac{M^2L_\lambda^3}{24\varepsilon}\right)^{1/3}
        \right\rceil,
        \left\lceil\frac{L_\lambda M}{\pi}\right\rceil
    \right\}.
\]
\end{theorem}

The use of Theorem~\ref{thm:marsh-beers-input} below is therefore conditional on remaining inside the fixed admissible class \(\mathcal G^r(Z^*,d)\). In particular, arbitrary landmark, polygonal, cubical, or image-domain reconstructions are not automatically covered by this theorem. They are covered only after one verifies that the reconstruction produces \(C^r\) embedded one-dimensional objects in the same fixed CW class and that the geometric metric \(d_{Z^*}\) is controlled.

\section{Probabilistic framework}

All random objects below are defined on a probability space \((\Omega,\mathcal F,\Prob)\). We assume throughout that the geometric and ECT error quantities appearing in probability statements are measurable random variables.

Let \(K\in\mathcal{C}\) be a fixed clean shape and let \(\widetilde K\) be a random perturbed shape of \(K\). In the cleanest version of the theory, \(\widetilde K\in\mathcal{C}\) almost surely. Define the random geometric perturbation radius
\[
    R = d_{\mathrm{geom}}(K,\widetilde K).
\]
If \eqref{eq:deterministic-stability} applies to the pair \((K,\widetilde K)\), then
\[
    d_{\ECT}(K,\widetilde K)
    \le
    \Phi(R)
\]
almost surely. Thus the ECT error is controlled by the random variable \(R\). The main question becomes probabilistic: how large is \(R\), and with what probability?

\begin{theorem}[Abstract probabilistic stability transfer]
\label{thm:abstract-transfer}
Let \(K\in\mathcal{C}\), and let \(\widetilde K\) be a random element of \(\mathcal{C}\). Assume the deterministic stability estimate \eqref{eq:deterministic-stability}. If \(r\geq 0\) and
\[
    \Prob(R>r) \le \psi(r),
    \qquad
    R=d_{\mathrm{geom}}(K,\widetilde K),
\]
then
\[
    \Prob\left(d_{\ECT}(K,\widetilde K)>\Phi(r)\right)
    \le
    \psi(r).
\]
Equivalently, if \(r_\delta\) satisfies \(\psi(r_\delta)\le\delta\), then with probability at least \(1-\delta\),
\[
    d_{\ECT}(K,\widetilde K)
    \le
    \Phi(r_\delta).
\]
\end{theorem}

\begin{proof}
By deterministic stability,
\[
    d_{\ECT}(K,\widetilde K)\le \Phi(R)
\]
almost surely. Since \(\Phi\) is nondecreasing, the event \(R\le r\) implies
\[
    d_{\ECT}(K,\widetilde K)\le \Phi(r).
\]
Therefore
\[
    \{d_{\ECT}(K,\widetilde K)>\Phi(r)\}
    \subseteq
    \{R>r\}.
\]
Taking probabilities gives the claim.
\end{proof}

\begin{theorem}[Open-threshold transfer]
\label{thm:open-transfer}
Let \(K\in\mathcal{C}\), and let \(\widetilde K\) be a random element of \(\mathcal{C}\). Suppose that for every \(\varepsilon>0\),
\[
    d_{\mathrm{geom}}(K,\widetilde K)<\varepsilon
    \quad\Longrightarrow\quad
    d_{\ECT}(K,\widetilde K)\le \Phi(\varepsilon).
\]
If
\[
    \Prob(R\geq r)\leq \psi(r),
    \qquad
    R=d_{\mathrm{geom}}(K,\widetilde K),
\]
then
\[
    \Prob\left(d_{\ECT}(K,\widetilde K)>\Phi(r)\right)
    \leq
    \psi(r).
\]
\end{theorem}

\begin{proof}
The event \(R<r\) implies \(d_{\ECT}(K,\widetilde K)\le \Phi(r)\). Hence
\[
    \{d_{\ECT}(K,\widetilde K)>\Phi(r)\}
    \subseteq
    \{R\geq r\}.
\]
Taking probabilities gives the claim.
\end{proof}

\begin{remark}[High-probability admissibility]
Some perturbation models may produce objects outside \(\mathcal{C}\) on a small exceptional event. In that case one can introduce an admissibility event \(A\), require \eqref{eq:deterministic-stability} on \(A\), and obtain the conservative bound
\[
    \Prob\left(A\cap\{d_{\ECT}(K,\widetilde K)>\Phi(r)\}\right)
    \le
    \Prob(A\cap\{R>r\}).
\]
Equivalently, if failure outside \(A\) is counted as an error event, then the bound becomes
\[
    \Prob\left(A^c\cup\left(A\cap\{d_{\ECT}(K,\widetilde K)>\Phi(r)\}\right)\right)
    \le
    \Prob(A^c)+\Prob(A\cap\{R>r\}).
\]
This observation is useful, but it should not be used to hide regularity assumptions. A concrete application still has to prove that \(A\) has high probability.
\end{remark}

\section{One-dimensional random perturbation models}

We now record two stochastic models that preserve the one-dimensional embedded-shape viewpoint. These statements should be read with a fixed deterministic stability input of the form \eqref{eq:deterministic-stability}; in the intended application this input is Theorem~\ref{thm:marsh-beers-input}, with \(d_{\mathrm{geom}}=d_{Z^*}\) and \(d_{\ECT}\) given by the Marsh--Beers ECT norm. The assumptions below are deliberately stated at the level of the geometric metric, because different reconstruction or extraction procedures lead to different geometric error estimates.

\subsection{Ambient displacement fields}

Let \(K\in\mathcal C\), and let \(\xi:K\to\R^d\) be a random displacement field. Define
\[
    \widetilde K
    =
    \{x+\xi(x):x\in K\}.
\]
Let \(A\) be the event that \(\widetilde K\in\mathcal C\) and that the deterministic stability theorem applies to the pair \((K,\widetilde K)\).

\begin{proposition}[Ambient displacement transfer]
\label{prop:ambient-transfer}
Assume that on the event \(A\),
\[
    d_{\mathrm{geom}}(K,\widetilde K)
    \le
    L_\xi\sup_{x\in K}\norm{\xi(x)}
\]
for some constant \(L_\xi>0\). If
\[
    \Prob\left(\sup_{x\in K}\norm{\xi(x)}>s\right)
    \le
    q(s),
\]
then for every \(r\geq 0\),
\[
    \Prob\left(A^c\cup\left(A\cap\{d_{\ECT}(K,\widetilde K)>\Phi(r)\}\right)\right)
    \le
    \Prob(A^c)+q(r/L_\xi).
\]
\end{proposition}

\begin{proof}
On \(A\), the event \(d_{\mathrm{geom}}(K,\widetilde K)>r\) implies
\[
    L_\xi\sup_{x\in K}\norm{\xi(x)}>r.
\]
Therefore
\[
    \Prob\bigl(A\cap\{d_{\mathrm{geom}}(K,\widetilde K)>r\}\bigr)
    \le
    q(r/L_\xi).
\]
On \(A\cap\{d_{\mathrm{geom}}(K,\widetilde K)\le r\}\), deterministic stability gives \(d_{\ECT}(K,\widetilde K)\le\Phi(r)\). Hence
\[
    \Prob\left(A^c\cup\left(A\cap\{d_{\ECT}(K,\widetilde K)>\Phi(r)\}\right)\right)
    \le
    \Prob(A^c)+q(r/L_\xi).
\]
\end{proof}

\begin{remark}
Proposition~\ref{prop:ambient-transfer} is a transfer statement, not a random-field concentration theorem. It does not by itself prove a supremum tail bound for a continuous random field. Such bounds require assumptions on the field, for example regularity, metric-entropy, or covariance assumptions in the Gaussian case. This is the main additional work needed if the ambient-field model is used as a principal example.
\end{remark}

\subsection{Finite landmark perturbations with reconstruction}

The finite landmark model gives a self-contained finite-dimensional stochastic transfer setting. Let \(x=(x_1,\dots,x_n)\in(\R^d)^n\), let
\[
    \widetilde x_i=x_i+\xi_i,
    \qquad i=1,\dots,n,
\]
and let \(F:(\R^d)^n\to\mathcal C\) be a reconstruction map. Set
\[
    K=F(x_1,\dots,x_n),
    \qquad
    \widetilde K=F(\widetilde x_1,\dots,\widetilde x_n).
\]
The Lipschitz condition imposed below is a deterministic regularity assumption on the reconstruction pipeline. This paper does not fix a particular reconstruction algorithm. Instead, any reconstruction method can be inserted into the theorem once it is known to be Lipschitz into the geometric metric used by the chosen deterministic ECT stability theorem. In applications, this assumption may be verified directly or obtained from local regularity on an admissible compact parameter region.

\begin{theorem}[Landmark perturbation stability]
\label{thm:landmark-transfer}
Assume the deterministic stability estimate \eqref{eq:deterministic-stability}. Suppose that \(F\) is \(L_F\)-Lipschitz into \((\mathcal C,d_{\mathrm{geom}})\) with respect to the landmark sup norm:
\[
    d_{\mathrm{geom}}(F(x),F(y))
    \le
    L_F\max_{1\le i\le n}\norm{x_i-y_i}
\]
for all landmark configurations \(x,y\in(\R^d)^n\). If the coordinate perturbations satisfy
\[
    \Prob(\norm{\xi_i}>s)\le q_i(s),
    \qquad
    i=1,\dots,n,
\]
then for every \(r\geq 0\),
\[
    \Prob\left(d_{\ECT}(K,\widetilde K)>\Phi(r)\right)
    \le
    \sum_{i=1}^n q_i(r/L_F).
\]
In particular, if \(q_i\le q\) for all \(i\), then
\[
    \Prob\left(d_{\ECT}(K,\widetilde K)>\Phi(r)\right)
    \le
    n\,q(r/L_F).
\]
\end{theorem}

\begin{proof}
By the Lipschitz assumption on \(F\),
\[
    d_{\mathrm{geom}}(K,\widetilde K)
    \le
    L_F \max_{1\le i\le n}\norm{\xi_i}.
\]
Hence
\[
    \{d_{\mathrm{geom}}(K,\widetilde K)>r\}
    \subseteq
    \bigcup_{i=1}^n \{\norm{\xi_i}>r/L_F\}.
\]
The union bound gives
\[
    \Prob\bigl(d_{\mathrm{geom}}(K,\widetilde K)>r\bigr)
    \le
    \sum_{i=1}^n q_i(r/L_F).
\]
Applying Theorem~\ref{thm:abstract-transfer} proves the result.
\end{proof}

\begin{corollary}[Bounded landmark perturbations]
\label{cor:landmark-bounded}
Under the hypotheses of Theorem~\ref{thm:landmark-transfer}, suppose
\[
    \norm{\xi_i}\le s_0
    \qquad
    \text{almost surely for every } i.
\]
Then
\[
    d_{\ECT}(K,\widetilde K)
    \le
    \Phi(L_Fs_0)
\]
almost surely.
\end{corollary}

\begin{proof}
The Lipschitz assumption gives \(d_{\mathrm{geom}}(K,\widetilde K)\le L_Fs_0\) almost surely. The deterministic stability estimate then gives the claim.
\end{proof}

\begin{corollary}[Sub-Gaussian landmark perturbations]
\label{cor:landmark-subgaussian}
Under the hypotheses of Theorem~\ref{thm:landmark-transfer}, suppose
\[
    \Prob(\norm{\xi_i}>s)\le C\exp(-a s^2)
\]
for all \(i\) and constants \(C,a>0\). Then for every \(0<\delta<nC\), with probability at least \(1-\delta\),
\[
    d_{\ECT}(K,\widetilde K)
    \le
    \Phi\left(
        L_F\sqrt{\frac{1}{a}\log\frac{nC}{\delta}}
    \right).
\]
\end{corollary}

\begin{proof}
Theorem~\ref{thm:landmark-transfer} gives
\[
    \Prob\left(d_{\ECT}(K,\widetilde K)>\Phi(r)\right)
    \le
    nC\exp\left(-a(r/L_F)^2\right).
\]
Choose \(r=L_F\sqrt{a^{-1}\log(nC/\delta)}\).
\end{proof}

\begin{corollary}[Sub-exponential landmark perturbations]
\label{cor:landmark-subexponential}
Under the hypotheses of Theorem~\ref{thm:landmark-transfer}, suppose
\[
    \Prob(\norm{\xi_i}>s)\le C\exp(-a s)
\]
for all \(i\) and constants \(C,a>0\). Then for every \(0<\delta<nC\), with probability at least \(1-\delta\),
\[
    d_{\ECT}(K,\widetilde K)
    \le
    \Phi\left(
        \frac{L_F}{a}\log\frac{nC}{\delta}
    \right).
\]
\end{corollary}

\begin{proof}
Theorem~\ref{thm:landmark-transfer} gives
\[
    \Prob\left(d_{\ECT}(K,\widetilde K)>\Phi(r)\right)
    \le
    nC\exp(-a r/L_F).
\]
Choose \(r=(L_F/a)\log(nC/\delta)\).
\end{proof}

\begin{corollary}[Polynomial-tail landmark perturbations]
\label{cor:landmark-polynomial}
Under the hypotheses of Theorem~\ref{thm:landmark-transfer}, suppose
\[
    \Prob(\norm{\xi_i}>s)\le C s^{-\alpha}
\]
for all \(i\) and constants \(C,\alpha>0\). Then for every \(0<\delta<1\), with probability at least \(1-\delta\),
\[
    d_{\ECT}(K,\widetilde K)
    \le
    \Phi\left(
        L_F\left(\frac{nC}{\delta}\right)^{1/\alpha}
    \right).
\]
\end{corollary}

\begin{proof}
Theorem~\ref{thm:landmark-transfer} gives
\[
    \Prob\left(d_{\ECT}(K,\widetilde K)>\Phi(r)\right)
    \le
    nC(r/L_F)^{-\alpha}.
\]
Choose \(r=L_F(nC/\delta)^{1/\alpha}\).
\end{proof}

\begin{corollary}[Marsh--Beers landmark stability]
\label{cor:mb-landmark}
Let \(Z^*\) be a fixed finite one-dimensional CW structure, and let \(F:(\R^d)^n\to\mathcal G^r(Z^*,d)\), \(r\geq 2\), be a reconstruction map. Let
\[
    K=F(x_1,\dots,x_n),
    \qquad
    \widetilde K=F(x_1+\xi_1,\dots,x_n+\xi_n).
\]
Assume that \(K\) has curvature bounded by \(M\), with one-cell arc lengths \(L_\lambda\), and assume that
\[
    d_{Z^*}(F(x),F(y))
    \le
    L_F\max_{1\le i\le n}\norm{x_i-y_i}
\]
for all landmark configurations under consideration. If
\[
    \Prob(\norm{\xi_i}\ge s)\le q_i(s),
    \qquad i=1,\dots,n,
\]
then for every \(\varepsilon>0\),
\[
    \Prob\left(
        \norm{\ECT_K-\ECT_{\widetilde K}}>
        \Phi_{MB}(\varepsilon)
    \right)
    \le
    \sum_{i=1}^n q_i(\varepsilon/L_F).
\]
\end{corollary}

\begin{proof}
If \(\max_i\norm{\xi_i}<\varepsilon/L_F\), then
\[
    d_{Z^*}(K,\widetilde K)<\varepsilon.
\]
The Marsh--Beers stability input, Theorem~\ref{thm:marsh-beers-input}, gives
\[
    \norm{\ECT_K-\ECT_{\widetilde K}}
    \le
    \Phi_{MB}(\varepsilon).
\]
Therefore
\[
    \left\{
        \norm{\ECT_K-\ECT_{\widetilde K}}>
        \Phi_{MB}(\varepsilon)
    \right\}
    \subseteq
    \bigcup_{i=1}^n\{\norm{\xi_i}\ge \varepsilon/L_F\}.
\]
The union bound completes the proof.
\end{proof}

\begin{corollary}[Explicit Marsh--Beers landmark tail bounds]
\label{cor:mb-landmark-tails}
In the setting of Corollary~\ref{cor:mb-landmark}, suppose additionally that the perturbations have a common tail bound
\[
    \Prob(\norm{\xi_i}\ge s)\le q(s),
    \qquad i=1,\dots,n.
\]
If \(\varepsilon_\delta\) satisfies \(nq(\varepsilon_\delta/L_F)\le\delta\), then with probability at least \(1-\delta\),
\[
    \norm{\ECT_K-\ECT_{\widetilde K}}
    \le
    \Phi_{MB}(\varepsilon_\delta).
\]
In particular, one may take the following choices whenever the displayed logarithms are positive:
\[
\varepsilon_\delta
=
\begin{cases}
L_F\sqrt{a^{-1}\log(nC/\delta)},&
q(s)=C\exp(-as^2),\\[0.4em]
(L_F/a)\log(nC/\delta),&
q(s)=C\exp(-as),\\[0.4em]
L_F(nC/\delta)^{1/\alpha},&
q(s)=Cs^{-\alpha}.
\end{cases}
\]
\end{corollary}

\subsection{A sufficient condition for Lipschitz reconstruction}

The Lipschitz reconstruction hypothesis is a standard consequence of uniform deterministic regularity on a finite-dimensional admissible parameter region. The following proposition records a typical situation without committing to a particular reconstruction algorithm.

\begin{proposition}[Smooth finite-dimensional reconstruction families]
\label{prop:smooth-reconstruction-lipschitz}
Let \(\Theta\subset \R^m\) be a compact convex parameter set and let
\[
    F:\Theta\to \mathcal C
\]
be a family of shapes in the deterministic stability class. Suppose there is a normed vector space \(B\), a map \(\Gamma:\Theta\to B\), and a constant \(C_F>0\) such that
\[
    d_{\mathrm{geom}}(F(\theta),F(\theta'))
    \le
    C_F\norm{\Gamma(\theta)-\Gamma(\theta')}_B
\]
for all \(\theta,\theta'\in\Theta\). If \(\Gamma\) is \(C^1\) on a neighbourhood of \(\Theta\) and
\[
    \sup_{\theta\in\Theta}\norm{D\Gamma(\theta)}_{\infty\to B}<\infty,
\]
then \(F\) is Lipschitz on \(\Theta\) with respect to the parameter sup norm. More precisely,
\[
    d_{\mathrm{geom}}(F(\theta),F(\theta'))
    \le
    C_F
    \left(\sup_{\vartheta\in\Theta}\norm{D\Gamma(\vartheta)}_{\infty\to B}\right)
    \norm{\theta-\theta'}_\infty .
\]
\end{proposition}

\begin{proof}
Since \(\Theta\) is convex, the segment between \(\theta\) and \(\theta'\) is contained in \(\Theta\). The Banach-space mean-value estimate gives
\[
    \norm{\Gamma(\theta)-\Gamma(\theta')}_B
    \le
    \left(\sup_{\vartheta\in\Theta}\norm{D\Gamma(\vartheta)}_{\infty\to B}\right)
    \norm{\theta-\theta'}_\infty .
\]
Combining this with the assumed comparison between \(d_{\mathrm{geom}}\) and the \(B\)-norm gives the claim.
\end{proof}

Consequently, if \(\widetilde\theta=\theta+\xi\) remains in \(\Theta\) almost surely and the scalar coordinate perturbations \(\xi_j\) satisfy tail bounds, the same union-bound argument as in Theorem~\ref{thm:landmark-transfer} applies with
\[
    L_F
    =
    C_F
    \sup_{\vartheta\in\Theta}\norm{D\Gamma(\vartheta)}_{\infty\to B}.
\]

In the Marsh--Beers setting, this example should be read with
\[
    F(\Theta)\subset \mathcal G^r(Z^*,d),
    \qquad r\geq 2,
\]
and with \(d_{\mathrm{geom}}=d_{Z^*}\). The compact admissible region \(\Theta\) must then encode the regularity needed to keep the reconstructed objects in the same \(C^r\) embedded one-dimensional class, with the relevant curvature and arc-length quantities controlled. This is the precise point at which a concrete reconstruction method would enter the theory.

\section{Consequences for unbounded perturbations}

The model-specific corollaries show why deterministic boundedness of the raw perturbation is not the essential requirement. What matters is a tail bound for the induced geometric error. In the landmark model, this geometric error is controlled by
\[
    L_F\max_{1\le i\le n}\norm{\xi_i}.
\]
Thus unbounded coordinate perturbations can still produce high-probability ECT stability as long as their tails are controlled.

For sub-Gaussian landmark noise, Corollary~\ref{cor:landmark-subgaussian} gives logarithmic dependence on \(1/\delta\) inside a square root. For sub-exponential landmark noise, Corollary~\ref{cor:landmark-subexponential} gives logarithmic dependence on \(1/\delta\). For polynomial-tail noise, Corollary~\ref{cor:landmark-polynomial} gives polynomial dependence on \(1/\delta\). In the Marsh--Beers setting, Corollary~\ref{cor:mb-landmark-tails} inserts these radii into the explicit modulus \(\Phi_{MB}\). The qualitative conclusion is that ECT stability under unbounded perturbations is a probabilistic statement: rare large geometric deviations are allowed, but their probabilities determine the final ECT error bound.

The ambient displacement model has the same structure once a supremum tail bound for \(\sup_{x\in K}\norm{\xi(x)}\) is available. The paper therefore separates the deterministic ECT stability theorem from the stochastic estimate needed for each perturbation model.

\section{Boundary with SAR-specific work}

The present paper is intentionally not a SAR image theorem. For SAR imagery, the basic multiplicative model is
\[
    Y(x)=X(x)\eta(x),
\]
and the logarithmic transform gives
\[
    \log Y(x)=\log X(x)+\log \eta(x).
\]
This suggests a possible route from multiplicative speckle to additive random perturbations. However, applying the ECT directly to SAR data requires choosing the mathematical input object. If a SAR pipeline extracts a one-dimensional geometric object, such as a target contour, skeleton, graph, or boundary curve, then the present geometric theory may apply after bounding the extraction-induced perturbation. If the input remains a grayscale image, range-Doppler map, binary mask, or cubical complex, then a separate deterministic stability theorem is needed for image-domain or cubical-complex ECTs.

Thus SAR-specific work has a different mathematical obligation: it must prove that the SAR observation and preprocessing pipeline induce a controlled random perturbation of the one-dimensional topological object to which the ECT is applied. Separate work can address cubical-complex or image-domain deterministic stability directly.

\section{Discussion and limitations}

The abstract transfer theorem is intentionally simple. Its role is to make explicit that probabilistic ECT stability can be obtained by combining two independent ingredients: a deterministic ECT stability theorem and a tail bound for the induced geometric error. The mathematically substantive work in an application lies in proving one or both of these ingredients.

The landmark perturbation result gives a finite-dimensional stochastic transfer theorem once the deterministic reconstruction estimate is supplied. Its assumptions are transparent: the reconstruction map must be Lipschitz into the geometric metric used by the deterministic ECT theorem, and the landmark perturbations must have controlled tails. The proof does not require independence; independence would only be needed for sharper probability estimates than the union bound.

We do not fix a particular reconstruction algorithm in this paper. Instead, reconstruction enters through the Lipschitz hypothesis in Theorem~\ref{thm:landmark-transfer}, allowing the transfer result to apply to any pipeline for which this deterministic geometric estimate is available. Proving such an estimate for a specific smooth reconstruction family is a natural way to specialize the framework, but it is separate from the probabilistic transfer argument.

The ambient displacement-field result is more general but less complete as an application theorem. To use it in a fully specified stochastic model, one must prove a supremum tail bound for the random field and verify that the perturbed embedded shape remains in the deterministic stability class with high probability. This is a natural next target, but it should not be folded into the present paper without explicit random-field assumptions.

For SAR, the present paper should be read as a geometric stability component rather than an image-domain theorem. Cubical-complex, binary-mask, grayscale-image, and range-Doppler versions of ECT stability require their own deterministic stability theory. Those cases are left for later work.

\bibliographystyle{plainnat}
\bibliography{references}

\end{document}